\documentclass{article}

\newcommand{\limp}{\Rightarrow}

\renewcommand{\lnot}{\neg\,}
\newcommand{\leT}{\le_{\mathrm{T}}}
\newcommand{\Phibar}{\overline\Phi}
\newcommand{\Phibarbar}{\overline{\Phibar}}
\newcommand{\Psitilde}{\widetilde\Psi}
\newcommand{\Mhat}{\widehat M}

\usepackage{amsthm}
\theoremstyle{definition}
\newtheorem{thm}{Theorem}[section]
\newtheorem{lem}[thm]{Lemma}
\newtheorem{cor}[thm]{Corollary}
\newtheorem{dfn}[thm]{Definition}
\newtheorem{rem}[thm]{Remark}

\usepackage[backref=page,colorlinks=true,allcolors=blue]{hyperref}

\begin{document}

\title{Comparing WO$(\omega^\omega)$ with $\Sigma^0_2$ induction}

\author{Stephen G. Simpson\\
  Department of Mathematics\\
  Pennsylvania State University\\
  University Park, PA 16802, USA\\
  \href{http://www.math.psu.edu/simpson}{http://www.math.psu.edu/simpson}\\
  \href{mailto:simpson@math.psu.edu}{simpson@math.psu.edu}}

\date{First draft: July 15, 2015\\
  This draft: July 27, 2015}

\maketitle

\begin{abstract}
  Let WO$(\omega^\omega)$ be the statement that the ordinal number
  $\omega^\omega$ is well ordered.  WO$(\omega^\omega)$ has occurred
  several times in the reverse-mathematical literature.  The purpose
  of this expository note is to discuss the place of
  WO$(\omega^\omega)$ within the standard hierarchy of subsystems of
  second-order arithmetic.  We prove that WO$(\omega^\omega)$ is
  implied by I$\Sigma^0_2$ and independent of B$\Sigma^0_2$.  We also
  prove that WO$(\omega^\omega)$ and B$\Sigma^0_2$ together do not
  imply I$\Sigma^0_2$.
\end{abstract}

\tableofcontents

\vfill

The author's research is supported by Simons Foundation Collaboration
Grant 276282.

\newpage

\section{Introduction}
\label{sec:intro}

In the language of second-order arithmetic, let WO$(\omega^\omega)$ be
the statement that $\omega^\omega$ is well ordered.\footnote{More
  precisely, WO$(\omega^\omega)$ is the statement that the standard
  set of Cantor normal form notations for the ordinal numbers less
  than $\omega^\omega$ is well ordered.}  In
\cite{hatz-ordinal,rmacc,hbt} it was shown that several theorems of
abstract algebra, including the Hilbert Basis Theorem, are
reverse-mathematically equivalent to WO$(\omega^\omega)$.  It is
therefore of interest to understand the place of WO$(\omega^\omega)$
within the usual hierarchy of subsystems of second-order arithmetic
\nocite{godel-fps}\cite{sosoa,gh}.

In this expository note we prove the following results.
\begin{itemize}
\item WO$(\omega^\omega)$ is provable from RCA$_0$ + $\Sigma^0_2$
  induction.
\item WO$(\omega^\omega)$ and $\Sigma^0_2$ bounding are independent of
  each other over RCA$_0$.
\item $\Sigma^0_2$ induction is not provable from RCA$_0$ +
  WO$(\omega^\omega)$ + $\Sigma^0_2$ bounding.
\end{itemize}
These results are perhaps well known and implicit in the literature on
fragments of arithmetic \cite{hajek-pudlak,kr-yo-bme-arxiv-v4}.  Our
reason for writing them up here is that, because of
\cite{hatz-ordinal,rmacc,hbt}, they deserve attention in the
reverse-mathematical context \cite{sosoa}.  I thank Keita Yokoyama for
explaining these results to me during a visit to Penn State, July
11--16, 2015.

\section{I$\Sigma^0_2$ implies WO$(\omega^\omega)$}
\label{sec:WO}

In this section we show that WO$(\omega^\omega)$ is provable in
RCA$_0$ + I$\Sigma^0_2$ but not in RCA$_0$ + B$\Sigma^0_2$.  Our
arguments in this section have a proof-theoretical flavor.

\begin{dfn}
  Let $\Phi$ range over $\Sigma^0_k$ formulas in the language of
  second-order arithmetic.  Note that $\Phi$ may contain free number
  variables and free set variables.  We consider the following
  schemes.
  \begin{enumerate}
  \item I$\Sigma^0_k$ is the \emph{$\Sigma^0_k$ induction} principle,
    i.e., the universal closure of
    \begin{center}
      $(\Phi(0)\land\forall i\,(\Phi(i)\limp\Phi(i+1)))\limp\forall
      i\,\Phi(i)$.
    \end{center}
  \item B$\Sigma^0_k$ is the \emph{$\Sigma^0_k$ bounding} principle,
    i.e., the universal closure of
    \begin{center}
      $(\forall i\,\exists j\,\Phi(i,j))\limp\forall m\,\exists
      n\,(\forall i<m)\,(\exists j<n)\,\Phi(i,j)$.
    \end{center}
  \end{enumerate}
  Note that I$\Sigma^0_k$ was called $\Sigma^0_k$-IND in \cite[Remark
  I.7.9]{sosoa}.  It is known that I$\Sigma^0_{k+1}$ implies
  B$\Sigma^0_{k+1}$ and B$\Sigma^0_{k+1}$ implies I$\Sigma^0_k$.
\end{dfn}

\begin{thm}
  WO$(\omega^\omega)$ is provable in RCA$_0$ + I$\Sigma^0_2$.
\end{thm}

\begin{proof}
  We reason in RCA$_0$ + I$\Sigma^0_2$.  Assume that $f$ is a
  descending sequence through $\omega^\omega$.  Consider the $\Pi^0_2$
  formula $\Phi(n,f)\equiv\forall\alpha\,($if $\exists
  i\,(f(i)<\alpha+\omega^n)$ then $\exists i\,(f(i)<\alpha))$.  By
  $\Pi^0_2$ induction on $n$ we prove $\forall n\,\Phi(n,f)$.
  Trivially $\Phi(0,f)$ holds.  Assume inductively that $\Phi(n,f)$
  holds, and let $\alpha$ be such that $\exists
  i\,(f(i)<\alpha+\omega^{n+1})$.  We then have $\exists m\,\exists
  i\,(f(i)<\alpha+\omega^n\cdot m)$, so by $\Pi^0_1$ induction there
  is a least such $m$.  If $m=0$ then $\exists i\,(f(i)<\alpha)$ and
  we are done.  If $m=l+1$ then $\exists i\,(f(i)<\alpha+\omega^n\cdot
  l+\omega^n)$, so by $\Phi(n,f)$ we have $\exists
  i\,(f(i)<\alpha+\omega^n\cdot l)$ contradicting our choice of $m$.
  We now see that $\forall n\,\Phi(n,f)$ holds.  For $\alpha=0$ this
  says that $\forall n\,($if $\exists i\,(f(i)<\omega^n)$ then
  $\exists i\,(f(i)<0))$, or in other words $\forall n\,\forall
  i\,(f(i)\ge\omega^n)$, contradicting the fact that
  $\omega^\omega=\sup_n\omega^n$.
\end{proof}

\begin{thm}
  WO$(\omega^\omega)$ is not provable in RCA$_0$ + B$\Sigma^0_2$.
\end{thm}

\begin{proof}
  It is known \cite[\S IX.3]{sosoa} that the provably total recursive
  functions of RCA$_0$ are just the primitive recursive functions.  In
  particular, totality of the Ackermann function is not provable in
  RCA$_0$.  It is also known \cite[Theorem
  IV.1.59]{hajek-pudlak}\footnote{See also the proofs of Lemma
    \ref{lem:is2} and Theorem \ref{thm:is2} below.} that RCA$_0$ +
  B$\Sigma^0_2$ is conservative over RCA$_0$ for $\Pi^0_2$ sentences.
  Therefore, totality of the Ackermann function is not provable in
  RCA$_0$ + B$\Sigma^0_2$.  On the other hand, totality of the
  Ackermann function is straightforwardly provable in RCA$_0$ +
  WO$(\omega^\omega)$.
\end{proof}

\begin{rem}
  More generally, for each $k\ge2$, letting $\omega_k=$ a stack of
  $\omega$'s of height $k$, it is known that WO$(\omega_k)$ is
  provable in RCA$_0$ + I$\Sigma^0_k$ and not provable in RCA$_0$ +
  B$\Sigma^0_k$.  These results belong to Gentzen-style proof theory.
\end{rem}

\section{WO$(\omega^\omega)$ does not imply B$\Sigma^0_2$}
\label{sec:bs2}

In this section we show that B$\Sigma^0_2$ is not provable in RCA$_0$
+ WO$(\omega^\omega)$.  Our arguments in this section and the next
have a model-theoretical flavor.

\begin{dfn}
  I$\Sigma_k$ and B$\Sigma_k$ consist of basic arithmetic plus the
  respective restrictions of I$\Sigma^0_k$ and B$\Sigma^0_k$ to the
  language of first-order arithmetic \cite{hajek-pudlak}.  It is known
  that I$\Sigma_{k+1}$ implies B$\Sigma_{k+1}$ and B$\Sigma_{k+1}$
  implies I$\Sigma_k$.
\end{dfn}

\begin{rem}
  \label{rem:unif}
  In the language of first-order arithmetic, let $\Phi(x)$ be a
  $\Sigma_{k+1}$ formula with a distinguished free variable $x$.
  Write $\Phi(x)$ as $\exists y\,\Theta(x,y)$ where $\Theta(x,y)$ is a
  $\Pi_k$ formula.  Let $\Phibar(x)$ be the $\Sigma_{k+1}$ formula
  \begin{center}
    $\exists z\,((z)_1=x\land\Theta((z)_1,(z)_2)\land\lnot(\exists
    w<z)\,\Theta((w)_1,(w)_2))$.
  \end{center}
  The universal closures of the following are provable in I$\Sigma_k$.
  \begin{enumerate}
  \item $\forall x\,(\Phibar(x)\limp\Phi(x))$.
  \item $\forall x\,\forall x'\,((\Phibar(x)\land\Phibar(x'))\limp
    x=x')$.
  \item\label{it:Phibar3} $(\exists x\,\Phi(x))\limp(\exists
    x\,\Phibar(x))$.
  \end{enumerate}
  Items 1 and 2 are trivial, and for item 3 we use I$\Sigma_k$ to
  prove the existence of $z$.  See also the discussion of ``special''
  $\Sigma_{k+1}$ formulas in \cite[\S IV.1(d)]{hajek-pudlak}.  The
  passage from $\Phi(x)$ to $\Phibar(x)$ will be referred to as
  \emph{uniformization with respect to} the variable $x$.
\end{rem}

\begin{lem}
  \label{lem:bs2}
  In the language of first-order arithmetic, let $\Psi$ be a $\Pi_3$
  sentence.  If I$\Sigma_1$ + $\Psi$ is consistent, then I$\Sigma_1$ +
  $\Psi$ does not prove B$\Sigma_2$.
\end{lem}

\begin{proof}
  Let $M$ be a nonstandard model of I$\Sigma_1$ + $\Psi$.  Fix a
  nonstandard element $c\in M$.  By Remark \ref{rem:unif} we know that
  every nonempty subset of $M$ which is $\Sigma_2(M)$-definable from
  $c$ contains an element which is $\Sigma_2(M)$-definable from $c$.
  Hence
  \begin{center}
    $M_2=\{x\in M\mid x$ is $\Sigma_2(M)$-definable from $c\}$
  \end{center}
  is a $\Sigma_2$-elementary submodel of $M$.  Therefore, since $\Psi$
  is a $\Pi_3$ sentence, $M_2$ satisfies $\Psi$.  And likewise, since
  I$\Sigma_1$ is axiomatized by $\Pi_3$ sentences, $M_2$ satisfies
  I$\Sigma_1$.  We shall finish the proof by showing that $M_2$ does
  not satisfy B$\Sigma_2$.

  Let $\Phi(e,x,c)$ be a $\Sigma_2$ formula which is \emph{universal}
  in sense that, as $e$ ranges over the natural numbers, $\Phi(e,x,c)$
  ranges over all $\Sigma_2$ formulas with one free variable $x$ and
  one parameter $c$.  For each $x\in M_2$ we know that $x$ is
  $\Sigma_2(M_2)$-definable from $c$, i.e., there exists a natural
  number $e$ such that $x$ is the unique element of $M_2$ such that
  $M_2$ satisfies $\Phi(e,x,c)$.  Moreover, since $e$ is a natural
  number and $c$ is nonstandard, we have $e<c$.  Uniformizing with
  respect to $x$, we see that $M_2$ satisfies $\Phibar(e,x,c)$ for all
  such pairs $e,x$.  Uniformizing again with respect to $e$, we see
  that for each $x\in M_2$ there is exactly one $e=e_x\in M_2$ such
  that $e<c$ and $M_2$ satisfies $\Phibarbar(e,x,c)$.  We now have a
  mapping $x\mapsto e_x$ which is $\Delta_2(M_2)$-definable from $c$
  and maps $M_2$ one-to-one into $\{e\in M_2\mid e<c\}$.  If $M_2$
  were a model of B$\Sigma_2$, then the restriction of $x\mapsto e_x$
  to $\{x\in M_2\mid x\le c\}$ would be $M_2$-finite, so we would have
  an $M_2$-finite mapping of the $M_2$-finite set $\{x\in M_2\mid x\le
  c\}$ into its $M_2$-finite proper subset $\{e\in M_2\mid e<c\}$.
  This contradiction shows that $M_2$ cannot satisfy B$\Sigma_2$.
\end{proof}

\begin{thm}
  \label{thm:bs2}
  In the language of second-order arithmetic, let $\exists X\,\forall
  Y\,\Psi(X,Y)$ be a $\Sigma^1_2$ sentence such that $\Psi(X,Y)$ is
  $\Pi^0_3$.  If RCA$_0$ + $\exists X\,\forall Y\,\Psi(X,Y)$ is
  consistent, then RCA$_0$ + $\exists X\,\forall Y\,\Psi(X,Y)$ does
  not prove B$\Sigma^0_2$.
\end{thm}

\begin{proof}
  Consider the $\Pi^0_3$ formula $\Psitilde(X)\equiv\forall Y\,(Y\leT
  X\limp\Psi(X,Y))$.  We may view $\Psitilde(X)$ as a $\Pi_3$ sentence
  in the language of first-order arithmetic with an extra unary
  predicate $X$.  Let $(M,X_M)$ be a nonstandard model of
  I$\Sigma_1(X)$ + $\Psitilde(X)$.  As in the proof of Lemma
  \ref{lem:bs2}, fix a nonstandard $c\in M$ and let $M_2=\{x\in M\mid
  x$ is $\Sigma_2(M,X_M)$-definable from $c\}$.  Also as in the proof
  of Lemma \ref{lem:bs2}, we have that $(M_2,X_M\cap M_2)$ satisfies
  I$\Sigma_1(X)$ + $\Psitilde(X)$ and does not satisfy B$\Sigma_2(X)$.
  Passing to the language of second-order arithmetic, it follows by
  \cite[\S IX.1]{sosoa} that $(M_2,\Delta_1(M_2,X_M\cap M_2))$
  satisfies RCA$_0$ + $\exists X\,\forall Y\,\Psi(X,Y)$ and does not
  satisfy B$\Sigma^0_2$.
\end{proof}

\begin{cor}
  \label{cor:bs2}
  RCA$_0$ + WO$(\omega^\omega)$ does not prove B$\Sigma^0_2$.  More
  generally, for any primitive recursive linear ordering $\alpha$ of
  the natural numbers, if RCA$_0$ + WO$(\alpha)$ is consistent then
  RCA$_0$ + WO$(\alpha)$ does not prove B$\Sigma^0_2$.
\end{cor}

\begin{proof}
  WO$(\alpha)$ can be written in the form $\forall Y\,\Psi(Y,Y)$ where
  $\Psi(X,Y)$ is as in the hypothesis of Theorem \ref{thm:bs2}.  Our
  corollary is then a special case of Theorem \ref{thm:bs2}.
\end{proof}

\section{WO$(\omega^\omega)$ + B$\Sigma^0_2$ does not imply I$\Sigma^0_2$}
\label{sec:is2}

In this section we show that I$\Sigma^0_2$ is not provable in RCA$_0$
+ WO$(\omega^\omega)$ + B$\Sigma^0_2$.

\begin{lem}
  \label{lem:is2}
  In the language of first-order arithmetic, let $\Psi$ be a $\Pi_3$
  sentence.  If B$\Sigma_2$ + $\Psi$ is consistent, then B$\Sigma_2$ +
  $\Psi$ does not prove I$\Sigma_2$.
\end{lem}

\begin{proof}
  Let $M$ be a nonstandard model of B$\Sigma_2$ + $\Psi$.  As in the
  proof of Lemma \ref{lem:bs2}, fix a nonstandard element $c\in M$ and
  consider the $\Sigma_2$-elementary submodel $M_2=\{a\in M\mid a$ is
  $\Sigma_2(M)$-definable from $c\}$.  We may safely
  assume\footnote{For instance, this would be the case if $M$ is
    countably saturated, or if $M$ satisfies I$\Sigma_2$.} that $M_2$
  is not cofinal in $M$.  We shall show that the submodel
  \begin{center}
    $\Mhat_2=\{x\in M\mid(\exists a\in M_2)\,(x<a)\}$
  \end{center}
  satisfies B$\Sigma_2$ + $\Psi$ + $\lnot\,$I$\Sigma_2$.

  Claim 1: $\Mhat_2$ is a $\Sigma_1$-elementary submodel of $M$.  To
  see this, let $\Phi(x)$ be a $\Sigma_1$ formula with no free
  variables other than $x$.  Given $u\in\Mhat_2$ such that $M$
  satisfies $\Phi(u)$, we need to show that $\Mhat_2$ satisfies
  $\Phi(u)$.  Write $\Phi(x)$ as $\exists y\,\Theta(x,y)$ where
  $\Theta(x,y)$ is $\Pi_0$.  Since $M_2$ satisfies I$\Sigma_1$, $M_2$
  satisfies $\forall a\,\exists b\,(\forall
  x<a)\,(\Phi(x)\limp(\exists y<b)\,\Theta(x,y))$.  Let $a\in M_2$ be
  such that $u<a$, and let $b\in M_2$ be such that $M_2$ satisfies
  $(\forall x<a)\,(\Phi(x)\limp(\exists y<b)\,\Theta(x,y))$.  Since
  $M_2$ is a $\Sigma_2$-elementary submodel of $M$, it follows that
  $M$ also satisfies $(\forall x<a)\,(\Phi(x)\limp(\exists
  y<b)\,\Theta(x,y))$.  In particular $M$ satisfies $(\exists
  y<b)\,\Theta(u,y)$, so let $v\in M$ be such that $v<b$ and $M$
  satisfies $\Theta(u,v)$.  Since $\Mhat_2$ is an initial segment of
  $M$, we have $v\in\Mhat_2$.  Moreover $\Mhat_2$ satisfies
  $\Theta(u,v)$, hence $\Mhat_2$ satisfies $\Phi(u)$, Q.E.D.

  Claim 2: $\Mhat_2$ satisfies B$\Sigma_2$.  To see this, assume that
  $\Mhat_2$ satisfies $\forall x\,\exists y\,\Phi(x,y)$ where
  $\Phi(x,y)$ is $\Sigma_2$ with parameters in $\Mhat_2$.  We need to
  show that $\Mhat_2$ satisfies $\forall a\,\exists b\,(\forall
  x<a)\,(\exists y<b)\,\Phi(x,y)$.  By means of a pairing function, we
  may safely assume that $\Phi(x,y)$ is $\Pi_1$ with parameters in
  $\Mhat_2$.  By Claim 1 we have $(\forall x\in\Mhat_2)\,(\exists
  y\in\Mhat_2)\,(M$ satisfies $\Phi(x,y))$.  Fix $a\in\Mhat_2$.  Since
  $\Mhat_2$ is an initial segment of $M$, we have $(\forall b\in
  M\setminus\Mhat_2)\,(M$ satisfies $(\forall x<a)\,(\exists
  y<b)\,\Phi(x,y))$.  And then, since $M$ satisfies I$\Sigma_1$, there
  is a least $b\in M$ such that $M$ satisfies $(\forall x<a)\,(\exists
  y<b)\,\Phi(x,y)$, and by underspill this least $b$ belongs to
  $\Mhat_2$.  Using Claim 1 again, we now see that $\Mhat_2$ satisfies
  $(\forall x<a)\,(\exists y<b)\,\Phi(x,y)$, Q.E.D.

  Claim 3: $\Mhat_2$ satisfies $\lnot\,$I$\Sigma_2$.  To see this,
  recall from the proof of Lemma \ref{lem:bs2} that the uniformizing
  formula $e<c\land\Phibar(e,x,c)$ gives a $\Sigma_2(M_2)$-definable
  mapping from a bounded subset of $M_2$ onto $M_2$.  Since $M_2$ is a
  cofinal $\Sigma_2$-elementary submodel of $\Mhat_2$, this same
  formula gives a $\Sigma_2(\Mhat_2)$-definable mapping from a bounded
  subset of $\Mhat_2$ onto an unbounded subset of $\Mhat_2$.  This
  implies that $\Mhat_2$ does not satisfy I$\Sigma_2$, Q.E.D.

  As a point of interest, note that our proofs of Claims 1 through 3
  used only the assumption that $M$ satisfies I$\Sigma_1$.  The
  assumption that $M$ satisfies B$\Sigma_2$ + $\Psi$ was not used in
  those proofs, but it will be used in the proof of Claim 4.

  Claim 4: $\Mhat_2$ satisfies $\Psi$.  To see this, write $\Psi$ as
  $\forall x\,\exists y\,\Phi(x,y)$ where $\Phi(x,y)$ is $\Pi_1$ with
  no free variables other than $x$ and $y$.  We need to show that
  $(\forall x\in\Mhat_2)\,(\exists y\in\Mhat_2)\,(\Mhat_2$ satisfies
  $\Phi(x,y))$.  By Claim 1 plus the fact that $M_2$ is cofinal in
  $\Mhat_2$, it will suffice to show that $(\forall a\in
  M_2)\,(\exists b\in M_2)\,(M$ satisfies $(\forall x<a)\,(\exists
  y<b)\,\Phi(x,y))$.  Fix $a\in M_2$.  Since $M$ satisfies B$\Sigma_2$
  + $\forall x\,\exists y\,\Phi(x,y)$, there exists $b$ in $M$ such
  that $M$ satisfies $(\forall x<a)\,(\exists y<b)\,\Phi(x,y)$.  But
  then, because $M_2$ is a $\Sigma_2$-elementary submodel of $M$ and
  $a$ belongs to $M_2$, there exists such a $b$ which also belongs to
  $M_2$, Q.E.D.
\end{proof}

\begin{thm}
  \label{thm:is2}
  In the language of second-order arithmetic, let $\exists X\,\forall
  Y\,\Psi(X,Y)$ be a $\Sigma^1_2$ sentence such that $\Psi(X,Y)$ is a
  $\Pi^0_3$ formula.  If RCA$_0$ + B$\Sigma^0_2$ + $\exists X\,\forall
  Y\,\Psi(X,Y)$ is consistent, then RCA$_0$ + B$\Sigma^0_2$ + $\exists
  X\,\forall Y\,\Psi(X,Y)$ does not prove I$\Sigma^0_2$.
\end{thm}

\begin{proof}
  As in the proof of Theorem \ref{thm:bs2}, we may view the $\Pi^0_3$
  formula $\Psitilde(X)\equiv\forall Y\,(Y\leT X\limp\Psi(X,Y))$ as a
  $\Pi_3$ sentence in the language of first-order arithmetic with an
  extra unary predicate $X$.  As in the proof of Lemma \ref{lem:is2},
  let $(M,X_M)$ be a nonstandard model of B$\Sigma_2(X)$ +
  $\Psitilde(X)$, fix a nonstandard $c\in M$, let $M_2=\{a\in M\mid a$
  is $\Sigma_2(M,X_M)$-definable from $c\}$, and let $\Mhat_2=\{x\in
  M\mid(\exists a\in M_2)\,(x<a)\}$.  Also as in the proof of Lemma
  \ref{lem:is2}, we have that $(\Mhat_2,X_M\cap\Mhat_2)$ satisfies
  B$\Sigma_2(X)$ + $\Psitilde(X)$ + $\lnot\,$I$\Sigma_2(X)$.  Passing
  to the language of second-order arithmetic, it follows by \cite[\S
  IX.1]{sosoa} that $(\Mhat_2,\Delta_1(\Mhat_2,X_M\cap\Mhat_2))$
  satisfies RCA$_0$ + B$\Sigma^0_2$ + $\exists X\,\forall
  Y\,\Psi(X,Y)$ + $\lnot\,$I$\Sigma^0_2$.
\end{proof}

\begin{cor}
  \label{cor:is2}
  RCA$_0$ + B$\Sigma^0_2$ + WO$(\omega^\omega)$ does not prove
  I$\Sigma^0_2$.  More generally, for any primitive recursive linear
  ordering $\alpha$ of the natural numbers, if RCA$_0$ + B$\Sigma^0_2$
  + WO$(\alpha)$ is consistent then RCA$_0$ + B$\Sigma^0_2$ +
  WO$(\alpha)$ does not prove I$\Sigma^0_2$.
\end{cor}

\begin{proof}
  WO$(\alpha)$ can be written as $\forall Y\,\Psi(Y,Y)$ where
  $\Psi(X,Y)$ is as in the hypothesis of Theorem \ref{thm:is2}.  Our
  corollary is then a special case of Theorem \ref{thm:is2}.
\end{proof}

\begin{rem}
  Theorems \ref{thm:bs2} and \ref{thm:is2} and Corollaries
  \ref{cor:bs2} and \ref{cor:is2} hold more generally, for all
  $k\ge2$, replacing $\Sigma^0_2$ by $\Sigma^0_k$ and $\Pi^0_3$ by
  $\Pi^0_{k+1}$, with essentially the same proofs.
\end{rem}

\phantomsection
\addcontentsline{toc}{section}{References}


\end{document}